\documentclass{amsart}
\usepackage{amsmath}
\usepackage{amssymb}
\usepackage{latexsym}
\newtheorem{theorem}{Theorem}[section]

\newtheorem{prop}[theorem]{Proposition}

\newtheorem{cor}[theorem]{Corollary}
\newtheorem{lemma}[theorem]{Lemma}

\newtheorem{prob}[theorem]{Problem}
\newtheorem{ques}[theorem]{Question}

\newtheorem{rem}[theorem]{Remark}

\newtheorem*{ack}{Acknowledgements}
\newtheorem{defn}[theorem]{Definition}
\newtheorem{remark}[theorem]{Remark}

\newcommand{\al}{\alpha}
\newcommand{\be}{\beta}
\newcommand{\g}{\gamma}

\newcommand{\e}{\varepsilon}

\newcommand{\s}{\sigma}

\newcommand{\G}{\Gamma}

\newcommand{\bbN}{\mathbb{N}}

\newcommand{\Span}{\operatorname{span}}
\providecommand{\FR}{\mathop{\rm FR}\nolimits}
\providecommand{\GL}{\mathop{\rm GL}\nolimits}

\providecommand{\supp}{\mathop{\rm supp}\nolimits}

\newcommand{\conv}{\operatorname{conv}}
\providecommand{\Isom}{\mathop{\rm Isom}\nolimits}

\newcommand{\lb}{\label}

\newcommand{\wtw}{if and only if }

\begin{document}

\dedicatory{Dedicated to the memory of Ted Odell}

\title[On Banach-Mazur problem and maximal norms]{On an isomorphic Banach-Mazur rotation problem  and maximal norms in Banach spaces} 



\author{S. J. Dilworth}
\address{
Department of Mathematics\\
University of South Carolina\\
Columbia, SC 29208, USA
}

\email{dilworth@math.sc.edu}

\author{B.~Randrianantoanina}

\address{
Department of Mathematics\\
Miami University\\
Oxford, OH 45056, USA}
\email{randrib@miamioh.edu}

%


\begin{abstract} 
We prove that  the spaces $\ell_p$, $1<p<\infty, p\ne 2$, and all  infinite-dimensional subspaces of their  quotient spaces do not  admit  equivalent almost transitive renormings. This  is a step towards the solution of the Banach-Mazur rotation problem, which asks whether a separable Banach space with a transitive norm has to be isometric or isomorphic to a Hilbert space.
We obtain this as a consequence of a new property of almost transitive spaces with a Schauder basis, namely we prove  that in such spaces the unit vector basis of $\ell_2^2$ belongs to the two-dimensional  asymptotic structure and we  obtain some information  about the  asymptotic structure in higher dimensions. 

 Further, we prove that the spaces $\ell_p$, $1<p<\infty$, $p\ne 2$, have continuum different renormings with 1-unconditional bases each with a different maximal isometry group, and that every symmetric space other than $\ell_2$ has at least a countable number of such renormings. On the other hand we show 
that  the  spaces $\ell_p$,  $1<p<\infty$, $p\ne 2$,  have  continuum different renormings each with an isometry group which is not  contained in any maximal bounded  subgroup  of the group of isomorphisms of $\ell_p$. 
\end{abstract} 
\maketitle

\section{Introduction}

The long-standing  Banach-Mazur rotation problem \cite{Banach} 
asks whether every separable Banach space with a transitive group of linear surjective  isometries is isometrically isomorphic to a Hilbert space. This problem has attracted a lot of attention in the literature (see a survey \cite{BGRP2002}) and there are several related open problems. In particular it is not known whether a separable  Banach space with a transitive group of isometries is  isomorphic to a Hilbert space, and, until now, it was unknown if such a space could be isomorphic to $\ell_p$ for some $p\ne 2$. We show in this paper that this is impossible and we provide further restrictions on classes of spaces that admit transitive or almost transitive equivalent norms  (a norm on $X$ is called {\it almost transitive} if the orbit under the group of isometries of any element $x$ in the unit sphere of $X$ is norm dense in the unit sphere of $X$). 

It is well-known  that  spaces $L_p[0,1]$ for $1<p<\infty$ are almost transitive.  In 1993 
Deville, Godefroy and Zizler \cite[p. 176]{DGZ} (cf. \cite[Problem~8.12]{FRmax}) asked whether every super-reflexive space admits an equivalent almost transitive norm.  Recently Ferenczi and Rosendal \cite{FRmax} answered this question negatively by exhibiting   a complex super-reflexive HI space which does not admit an equivalent  almost transitive renorming. They  asked   \cite[Remark after Theorem~7.5]{FRmax} whether an example can be found among more classical Banach spaces.

The first main result of this paper is that  the spaces $\ell_p$, $1<p<\infty, p\ne 2$, and all  infinite-dimensional subspaces of their  quotient spaces do not  admit  equivalent almost transitive renormings.  We also prove that all infinite-dimensional subspaces of  Asymptotic-$\ell_p$ spaces for  $1\le p<\infty, p\ne 2$, fail to admit  equivalent almost transitive norms (Theorem~\ref{thm: Asymptoticlp}). Moreover we give an example of a super-reflexive space which does not contain either an  Asymptotic-$\ell_p$  space or  a subspace  which admits an
almost transitive norm (Corollary~\ref{example}).

We combine our results with known results about the structure of subspaces of $L_p$, $2<p<\infty$, and we obtain that if $X$ is a subspace of $L_p$, $2<p<\infty$, or, more generally, of any non-commutative $L_p$-space  \cite{HJX10} for  $2<p<\infty$, such that  every subspace of $X$ admits an equivalent almost transitive norm, then $X$ is isomorphic to  a Hilbert space (Corollary~\ref{subspaceofLp}). The same result is also true for 
 subspaces of the Schatten class $S^p(\ell_2)$ for  $1<p<\infty$, $p\ne 2$, (Corollary~\ref{schatten}). 
This suggests the following question.

\begin{prob}  Suppose that every subspace of a Banach space $X$ admits  an equivalent almost transitive renorming. Is $X$ isomorphic to  a Hilbert space?
\end{prob}

As additional information related  to this problem, we prove that if such a space $X$ is isomorphic to a stable Banach space then $X$  is $\ell_2$-saturated (Corollary~\ref{pin12} and Remark~\ref{stable}).

Our results are a consequence of a new property of almost transitive spaces with a Schauder basis. Namely we prove  that in such spaces the unit vector basis of $\ell_2^2$ belongs to the two-dimensional  asymptotic structure. We also obtain some information  about the  asymptotic structure in higher dimensions. Our method relies on an application of the classical Dvoretzky theorem and it  enables us to give up to date the most significant restrictions on isomorphic classes  of spaces which admit a transitive or almost transitive norm. In particular we obtain estimates for the power types of the upper and lower envelopes of $X$ and of $p$ and $q$ in $(p,q)$-estimates of $X$ (Corollaries~\ref{cor: upperenvelope} and \ref{cor: pqestimates}). From this we obtain a version of Krivine's theorem
for spaces with asymptotic unconditional structure and a subspace which admits an  almost transitive norm (Theorem~\ref{upper2}). Another consequence  is a characterization of subspaces of Orlicz sequence spaces which contain a subspace which admits an equivalent almost transitive norm (Theorem~\ref{orlicz}).

The second part of the paper is devoted to the study of maximal renormings on Banach spaces. For a Banach space  $(X,\|\cdot\|)$ we denote by $\Isom(X,\|\cdot\|)$ the group of all linear surjective isometries of $X$. We say that a Banach space  $(X,\|\cdot\|)$ is  {\it maximal} if whenever $|\!|\!|\cdot|\!|\!|$ is an equivalent norm on $X$ such that $\Isom(X,\|\cdot\|)\subseteq\Isom (X,|\!|\!|\cdot|\!|\!|)$, then $\Isom (X,\|\cdot\|)=\Isom (X,|\!|\!|\cdot|\!|\!|)$.
This notion was introduced by Pe\l czy\'nski and Rolewicz \cite{PR62}, cf. \cite{Rol}, and has been extensively studied, see e.g. a  survey \cite{BGRP2002}. Rolewicz  \cite{Rol} proved that all non-hilbertian 1-symmetric spaces  are maximal. In 1982 Wood \cite{W82} posed a problem whether every Banach space admits an equivalent maximal norm. 

The study of isometry groups of renormings of $X$ is equivalent to the study of bounded subgroups 
 of the group $\GL(X)$ of all isomorphisms from $X$ onto $X$. Indeed, if $G$ is a bounded subgroup of $\GL(X)$ we can define an equivalent norm $\|\cdot\|_G$ on $X$ by
$$\|x\|_G=\sup_{g \in G}\|gx\|.$$
Then $G$ is a subgroup of $\Isom(X,\|\cdot\|_G)$. Thus Wood's problem asks whether for every space $X$, $\GL(X)$ has a maximal bounded subgroup.

 In 2013 Ferenczi and Rosendal \cite{FRmax} answered Wood's problem negatively by  exhibiting   a complex super-reflexive space and a real reflexive space, both without a maximal bounded subgroup of the isomorphism group. 

In 2006 Wood \cite{Wood}, cf. \cite[p. 200]{FJiso2} and \cite[p.~1774]{FRmax}, asked, what he called a more natural question, whether for every Banach space there exists an equivalent maximal renorming  whose isometry group contains the original isometry group, i.e. whether every bounded subgroup of $\GL(X)$ is contained in a maximal bounded subgroup of $\GL(X)$.  As elaborated by Wood \cite{Wood} and Ferenczi and Rosendal \cite{FRmax} this question is related to the Dixmier's unitarisability problem whether every countable group all of whose bounded representations on a Hilbert space are unitarisable is amenable, see also \cite{P05}.

There are several known groups with bounded representations which are not unitarizable; however it is unknown whether isometry groups of renormings of Hilbert space induced by these representations are maximal or even contained in a maximal isometry group of another equivalent renorming. The answer to this question would elucidate the following portion of the Banach-Mazur problem.
\begin{prob} (see \cite[Problems 1.2 and 8.7]{FRmax}) Suppose that $|\!|\!|\cdot|\!|\!|$ is an equivalent maximal norm on a Hilbert space $\mathcal{H}$. Is $|\!|\!|\cdot|\!|\!|$ necessarily Euclidean?
\end{prob}
 
In this paper we show that even in Banach spaces  which have a maximal bounded subgroup of $\GL(X)$ there can also exist bounded subgroups of $\GL(X)$  which are not contained in any maximal bounded subgroup of $\GL(X)$. We prove that  $\ell_p$, $1 < p<\infty$, $p\ne2$, the 
2-convexified Tsirelson space $T^{(2)}$, and the space $U$ with a universal unconditional basis, have a  continuum of  renormings  none of whose isometry groups  is  contained in any maximal  bounded subgroup of $\GL(X)$. We note that $T^{(2)}$  is a weak Hilbert space. We do not know whether $T^{(2)}$ or general weak Hilbert spaces, other than $\ell_2$, have a maximal  bounded subgroup of $\GL(X)$. 

We also study maximal  bounded subgroups of $\GL(X)$ for   Banach spaces $X$ with 1-unconditional bases. We prove that $\ell_p$, $1 < p<\infty$, $p\ne 2$, and $U$ have continuum different renormings with 1-unconditional bases each with a different maximal isometry group, and that every  symmetric space other than $\ell_2$ has at least a countable number of such renormings. 
As mentioned above it is unknown whether the Hilbert space has a  unique, up to conjugacy,  maximal bounded subgroup of $\GL(X)$. 
Motivated by our results we  ask the following question.
\begin{ques}
Does there exist a separable Banach space $X$  with a unique, up to conjugacy,   maximal  bounded subgroup of $\GL(X)$? If yes, does $X$ have to be 
 isomorphic to a Hilbert space?
\end{ques}

\section{Almost and convex transitivity}

Throughout this section $X$ and $Y$ will denote  real or complex infinite dimensional 
  Banach spaces   and the term   \textit{subspace}
always  means  a closed \textit{infinite-dimensional} linear subspace. 

If a Banach space $X$ has  a  Schauder basis, a normalized Schauder basis will be denoted $(e_i)$ and its  biorthogonal sequence will be denoted $(e_i^*)$.
 For $n \ge 1$,   $P^X_n$ will denote the basis projection from $X$ onto the linear span of $e_1,\dots,e_n$. The \textit{support} of $x \in X$ is defined by $\supp x =\{i \in \mathbb{N} \colon e_i^*(x) \ne 0\}$. If $x,y \in X$ we write $x < y$ if
$\max \supp x < \min \supp y$ and $x > N$ if $\supp x \subset [N+1,\infty)$. We say that a sequence $(x_i)$ of  vectors is a \textit{normalized block basis} if $x_1 < x_2 < x_3<\dots$ and $\|x_i\|=1$ ($i\ge1$).

\begin{lemma} \label{lem: circle} Suppose that $X$ has a Schauder basis and contains a subspace $Y$ which is almost transitive.  Given $\delta>0$, $y_0 \in Y$, and $N \in \bbN$  there exists $y \in Y$, with $\|P^X_N(y)\| <\delta$, such that for all scalars $a,b$,  we have
\begin{equation} \label{eq: circle1} (1-\delta)(|a|^2\|y_0\|^2 + |b|^2)^{1/2} \le \left\|ay_0+by\right\| \le (1+ \delta)(|a|^2\|y_0\|^2 + |b|^2)^{1/2}. \end{equation} 
\end{lemma}
\begin{proof}   By compactness there exists $n := n(N,\delta)$ such that if $(y_i)_{i=1}^n \subset B_Y$
then there exist $1 \le i < j \le n$ such that  $\|P^X_N(y_j-y_i)\| < \delta$. 
 By Dvoretzky's theorem \cite{D61} and almost transitivity of $Y$, there exist 
$(y_i)_{i=1}^n \subset S_Y$ such that for all scalars $a,b_1,\dots,b_n$, we have
\begin{equation} \label{eq: dvoretzky} 
\begin{split} 
(1 - \delta)\left(|a|^2\|y_0\|^2 + \sum_{i=1}^n |b_i|^2\right)^{1/2} &\le \left\|ay_0 + \sum_{i=1}^n b_i y_i\right\| \\ 
&\le (1+ \delta)\left(|a|^2\|y_0\|^2 + \sum_{i=1}^n |b_i|^2\right)^{1/2}.
\end{split}
\end{equation}
 Choose $1 \le i < j \le n$ such that $\|P^X_N(y_j-y_i)\| < \delta$ and set $y = (1/\sqrt2)(y_j-y_i)$. Then  \eqref{eq: circle1} follows from \eqref{eq: dvoretzky}. 
\end{proof}
\begin{rem} 
For a Banach space $X$, let
$$ \FR(X):= \{1 \le r \le \infty \colon \ell_r \text{ is finitely representable in $X$}\}.$$
 Suppose that $r \in \FR(Y)$, where  $Y$ is as in Lemma~\ref{lem: circle}.   The proof yields (after the obvious modifications)  the same conclusion except \eqref{eq: circle1} should be replaced by
\begin{equation} \label{eq: circle2} (1-\delta)(|a|^r\|y_0\|^r + |b|^r)^{1/r} \le \left\|ay_0+by\right\| \le (1+ \delta)(|a|^r\|y_0\|^r + |b|^r)^{1/r} \end{equation} with the obvious modification for $r = \infty$. In particular, by the Maurey-Pisier theorem  \cite{MP}, this holds for all $r\in [p_Y,2] \cup \{q_Y\}$,
where $$p_Y := \sup\{ 1 \le p \le 2 \colon \text{$Y$ has type $p$}\}$$ and $$ q_Y := \inf\{ 2 \le q < \infty \colon \text{$Y$ has cotype $q$}\}.$$
\end{rem}
\begin{theorem} \label{prop: blockbasis}  
Suppose that $X$ has a Schauder basis and contains a subspace $Y$ which is almost transitive.  Let $r \in \FR(Y)$.
Then, given $\varepsilon>0$ and any sequence $(a_i)$  of nonzero scalars, there exists a normalized block basis $(x_i)$ in $X$
such that, for all $m \ge 1$ and all scalars $b$, we have 
\begin{equation} \label{eq: blockbasis} \begin{split}
(1-\varepsilon)\left(\sum_{k=1}^{m} |a_k|^r + |b|^r\right)^{1/r} &\le \left\|\sum_{k=1}^m a_k x_k + bx_{m+1}\right\| \\ &\le (1+\varepsilon)\left(\sum_{k=1}^{m} |a_k|^r + |b|^r\right)^{1/r}.\end{split} \end{equation} 
\end{theorem}
\begin{proof} First we prove the result for $r=2$.
Let $(\delta_i)_{i=1}^\infty$  be a  (sufficiently small) positive decreasing sequence. 
 We construct $(x_i)$ and an auxiliary sequence $(y_i) \subset Y$ iteratively. Let $y_1 \in S_Y$ be chosen arbitrarily.  Chose a finitely supported vector  $x_1 \in S_X$  such that $\|y_1 - x_1\| < \delta_1$. Let $(\varepsilon_i)$ be a strictly increasing sequence satisfying $0 < \varepsilon_i < \varepsilon/3$
for all $i$. Suppose $m \ge 1$ and that $y_i \in Y$ and finitely supported $x_i \in S_X$  ($1 \le i \le m$) have been chosen such that
 $x_1<x_2<\dots<x_m$, $\|x_i - y_i\| < \delta_i$, and, for all $1 \le j \le m$ and  scalars $b$,
\begin{equation} \label{eq: theyvectors} \begin{split}
(1-\varepsilon_j)\left(\sum_{k=1}^{j-1} |a_k|^2 + |b|^2\right)^{1/2} &\le \left\|\sum_{k=1}^{j-1} a_k y_k + by_{j}\right\|\\ &\le (1+\varepsilon_j)\left(\sum_{k=1}^{j-1} |a_k|^2 + |b|^2\right)^{1/2}. \end{split} \end{equation} 
 Let $N := \max \supp(x_m)$ and let $\delta>0$ be sufficiently small.
By  Lemma~\ref{lem: circle} applied to $y_0 = \sum_{k=1}^m a_k y_k$, there exists $y_{m+1} \in Y$ such that
$\|P^X_N(y_{m+1})\| < \delta$ and, for all scalars $a,b$, 
\begin{equation} \lb{eq: circled}
(1-\delta)(|a|^2\|y_0\|^2 + |b|^2)^{1/2} \le \|ay_0+by_{m+1}\| \le (1+ \delta)(|a|^2\|y_0\|^2 + |b|^2)^{1/2}.  
\end{equation} 
In particular, for all scalars $b$, we have
\begin{align*}
(1-\varepsilon_m)(1-\delta)\left(\sum_{k=1}^{m} |a_k|^2 + |b|^2\right)^{1/2} &\le \left\|\sum_{k=1}^m a_k y_k + by_{m+1}\right\| \\ &\le (1+\varepsilon_m)(1+\delta)\left(\sum_{k=1}^{m} |a_k|^2 + |b|^2\right)^{1/2},
 \end{align*}
and hence \eqref{eq: theyvectors} holds for $n=m+1$ provided  
$\delta(1+ \varepsilon_m)< \varepsilon_{m+1}-\varepsilon_m$, which we may assume. Set $\tilde x_{m+1} := P^X_{N_1}(y_{m+1}) - P^X_N(y_{m+1})$, where $N_1$ is chosen sufficiently large to ensure that $\|y_{m+1} - P^X_{N_1}(y_{m+1})\| < \delta$. Then $x_m<\tilde x_{m+1}$ and
$$\|\tilde x_{m+1} - y_{m+1}\| \le \|y_{m+1} - P^X_{N_1}(y_{m+1})\|  + \|P^X_N(y_{m+1})\| < 2\delta.$$
By \eqref{eq: circled} $|1 - \|y_{m+1}\|| \le \delta$, and hence $|1-\|\tilde x_{m+1}\|| < 3\delta$. 
Let $x_{m+1} = \tilde x_{m+1}/\|\tilde x_{m+1}\|$. Then $\|x_{m+1}\|=1$ and 
$$\|x_{m+1}-y_{m+1}\| \le \|x_{m+1}-\tilde x_{m+1}\| + \|\tilde x_{m+1} - y_{m+1}\| < 3\delta + 2\delta = 5\delta.$$ 
Nence $\|x_{m+1} - y_{m+1}\| < \delta_{m+1}$ provided $5\delta< \delta_{m+1}$, which we may assume. This completes the proof of the inductive step.

Finally, provided 
$(\sum_{i=1}^\infty |a_i| \delta_i) < |a_1| \varepsilon/3$ and $\delta_1 < \varepsilon/3$, which we may assume, \eqref{eq: blockbasis} follows from \eqref{eq: theyvectors} by an easy triangle inequality calculation and the fact that $\|x_i-y_i\|<\delta_i$ for all $i\ge1$.

The proof for $r\in \FR(Y)$ is very similar except  \eqref{eq: circle2} is used instead of \eqref{eq: circle1}.
\end{proof}

Since the behaviour of block bases in $\ell_p$  and $c_0$ is well understood, as an immediate consequence we obtain

\begin{theorem} \lb{lpandtsirelson}
No subspace of $\ell_p$, $1\le p<\infty$, $p\ne 2$,  or of $c_0$ admits an almost transitive renorming.
\end{theorem}

\begin{proof} 
 Let $Y$ be a subspace of $X$ so that $Y$ admits an equivalent almost transitive norm $|\!|\!|\cdot|\!|\!|$. It is well-known that any equivalent norm 
on a subspace may be extended to an equivalent norm on the whole space, see e.g. \cite[p.\ 55]{FHHSPZ}. So $|\!|\!|\cdot|\!|\!|$ extends to an equivalent norm
$|\!|\!|\cdot|\!|\!|$ on $X$. 
By Theorem~\ref{prop: blockbasis}
for any $\e>0$ there exists    a normalized block basis $(x_k)$ in $X$
such that  for all $n\in\bbN$, 
\begin{equation} \label{eq: blockbasis2} 
(1-\varepsilon)n^{1/2} \le\Big|\!\Big|\!\Big|\sum_{k=1}^{n}  x_k\Big|\!\Big|\!\Big| \le (1+\varepsilon)n^{1/2}. \end{equation}

It is well known that when $X=\ell_p$, $1\le p<\infty$, every block basis is isometrically equivalent to the standard basis of $\ell_p$,
 see e.g. \cite{LT}. 
Since $|\!|\!|\cdot|\!|\!|$ is $C$-equivalent to $\|\cdot\|_{\ell_p}$, $n$ is arbitrary  and $p\ne 2$, we obtain a contradiction.
The proof for $c_0$ is similar.
\end{proof}

\begin{remark} \label{rem:credit} 
F. Cabello-Sanchez \cite{CS98} proved that almost transitive Banach
spaces which  are either Asplund or have the Radon-Nikodym property
 are actually super-reflexive, and thus, in particular, it was known that  $\ell_1$ and $c_0$ and their subspaces do not admit an equivalent almost transitive norm. We note however that there do exist spaces without the Radon-Nikodym property which are  almost transitive, e.g. it is well known that $L_1[0,1]$ is  almost transitive. Also Lusky \cite{L79} proved that   every separable Banach space $(X,\|\cdot\|)$ is  complemented in a separable almost transitive space $(Y,\|\cdot\|)$, its norm being an extension of the norm on $X$.  
\end{remark}

\begin{remark} It is instructive to observe that  Theorem~\ref{prop: blockbasis}, when applied to the Haar basis of $L_p[0,1]$,
does not contradict the fact that $L_p[0,1]$ is almost transitive. This is simply because the unit vector basis of $\ell_2$ is
($1+\varepsilon$)-equivalent to a block basis of the Haar basis. We thank the anonymous referee for this remark. \end{remark}

\begin{remark} \label{rem: game} 
It is clear from the proof of Theorem~\ref{prop: blockbasis} that if we consider the corresponding \textit{infinite asymptotic game} in $X$ then the vector player has a winning strategy. More precisely, suppose that $\varepsilon>0$ and $(a_n)$ are fixed. Then $\forall n_1\, \exists x_1 > n_1$ s.t. $\forall n_2\, \exists x_2 > n_2\dots$
such that the \textit{outcome} $(x_i)$ satisfies \eqref{eq: blockbasis}. 
\end{remark}

We recall the notion of asymptotic structure introduced by Maurey, Milman, and Tomczak-Jaegermann \cite{MMT}. A basis $(b_i)_{i=1}^n$ of unit vectors for an $n$-dimensional normed space belongs to $\{X, (e_i)\}_n$
if, given $\varepsilon>0$, the second player has a winning strategy in the asymptotic game to produce a sequence $(x_i)_{i=1}^n$ that is $(1+\varepsilon)$-equivalent to $(b_i)_{i=1}^n$.
Precisely, 
$\forall m_1\, \exists x_1>m_1$ s.t. $\forall m_2\, \exists x_2 > m_2$ s.t. $\dots \forall m_n\, \exists x_n > m_n$ such that
there exist $c$ and $C$, with $0<c\le C$ and $C/c < 1+\varepsilon$, such that for all scalars $(a_i)_{i=1}^n$, we have  
$$c\left\|\sum_{i=1}^n a_i b_i\right\| \le \left\|\sum_{i=1}^n a_i x_i \right\| \le C \left\|\sum_{i=1}^n a_i b_i\right\|.$$ 
\begin{theorem} \lb{thm: asymptoticstructure}
Suppose that $X$ has a Schauder basis and contains a subspace $Y$ which is almost transitive. Let $r \in \FR(Y)$.
 Then, for all $n \ge 2$ and for all  nonzero scalars $a_1,\dots,a_{n-1}$, there exists  
$(b_i)_{i=1}^n \in \{X, (e_i)\}_n$ such that for all scalars $\lambda$ and for all $1 \le k < n$, we have
\begin{equation}
\left\|\sum_{i=1}^k a_i b_i  + \lambda b_{k+1}\right\| = \left(\sum_{i=1}^k |a_i|^r + |\lambda|^r\right)^{1/r}.  
\end{equation}
In particular, the unit vector basis of $\ell_r^2$ belongs to $\{X, (e_i)\}_2$.
\end{theorem} 
\begin{proof} By a theorem of Knaust, Odell, and Schlumprecht  \cite{KOS} (cf. \cite{O}), given $\varepsilon_n \downarrow 0$, the basis $(e_i)$ may be \textit{blocked} as $(E_k)$,
where $E_k = \operatorname{span} \{e_i \colon m_{k} \le i < m_{k+1}\}$ so that, for all $n \ge 1$,  if $m_n \le x_1 < x_2 < \dots < x_n$ is a \textit{skipped} sequence of unit vectors
with respect to the blocking $(E_k)$ (i.e., if $i < j$ then there 
 exists $k$ such that $x_i < m_k \le m_{k+1} \le x_j$), then $(x_i)_{i=1}^n$ is
$(1+\varepsilon_n)$-equivalent to some $(b_i)_{i=1}^n \in \{X, (e_i)\}_n$. The result now follows from the proof of Theorem~\ref{prop: blockbasis} and Remark~\ref{rem: game}.
\end{proof}
\begin{rem} Theorem~\ref{thm: asymptoticstructure} remains valid for the asymptotic structure associated to the collection $\mathfrak{B}^0(X)$ of subspaces of $X$ of finite codimension (see \cite[Section~1.1]{MMT}), and in that case it is not necessary to assume that $X$ has a basis.  In fact, Theorem~\ref{thm: asymptoticstructure} can also be proved using \cite[Section~1.5]{MMT} instead of \cite{KOS}.
\end{rem}

Recall that $X$ is an asymptotic-$\ell_p$ space ($1 \le p < \infty$) in the sense of \cite{MMT} if there exists $C>0$ such that for all $n \ge 1$, for all 
$(b_i)_{i=1}^n \in \{X, (e_i)\}_n$, and for all scalars $(a_i)_{i=1}^n$, we have \begin{equation} \lb{eq: Asymptoticlp}
 \frac{1}{C} \left(\sum_{i=1}^n |a_i|^p\right)^{1/p} \le \left\|\sum_{i=1}^n a_i b_i \right\| \le C\left(\sum_{i=1}^n |a_i|^p\right)^{1/p}, \end{equation}
and that $X$ is an asymptotic $\ell_\infty$-space if (in place of \eqref{eq: Asymptoticlp}) we have
\begin{equation} \lb{eq: Asymptoticlinfty} \frac{1}{C} \max_{1 \le i \le n} |a_i| \le \left\|\sum_{i=1}^n a_i b_i \right\| \le C\max_{1 \le i \le n}|a_i|. \end{equation}
We shall follow \cite[p.\ 238]{O} in calling such spaces  Asymptotic-$\ell_p$ spaces  (with a capital A) in order to  to distinguish them from a smaller class of spaces which are also known, somewhat confusingly,  as asymptotic-$\ell_p$ spaces \cite[p.\ 226]{O}. This narrower  notion of asymptotic-$\ell_p$ space is defined by the  following condition on finite block sequences: there exists $C>0$ such that for all $n  \le x_1<\dots<x_n$, we have
$$ \frac{1}{C} \left(\sum_{k=1}^n \|x_k\|^p\right)^{1/p} \le \left\| \sum_{k=1}^n  x_k \right\| \le C\left(\sum_{k=1}^n \|x_k\|^p\right)^{1/p},$$ with the obvious modification for $p =\infty$.
The latter narrower class of asymptotic-$\ell_p$ space contains, e.g., the $p$-convexified Tsirelson space $T^{(p)}$ for $1 \le p < \infty$
(see \cite{CJT84} and  \cite{CS89}), while $T^*$, the dual of the Tsirelson space $T$ ($=T^{(1)}$), is an asymptotic-$\ell_\infty$ space.
 
\begin{theorem} \lb{thm: Asymptoticlp}
 Suppose $X$ is an Asymptotic-$\ell_p$ space, $1 \le p \le \infty$, $p \ne 2$. Then no subspace $Y$ of $X$ admits an equivalent almost transitive norm. \end{theorem} \begin{proof} Suppose $X$ contained such a subspace $Y$. By Theorem~\ref{thm: asymptoticstructure}, for all $n \ge 1$ there exists $(b_i)_{i=1}^n \in \{X, (e_i)\}_n$ such that $\|\sum_{i=1}^n b_i\| = \sqrt{n}$. But this contradicts
\eqref{eq: Asymptoticlp} (or \eqref{eq: Asymptoticlinfty} if $p=\infty$) when $n$ is sufficiently large.
\end{proof}

For $1<p<\infty$, let $\mathcal{C}_p$ denote the class of Banach spaces $X$ which are isomorphic to a subspace 
 of an $\ell_p$-sum of finite-dimensional normed spaces. It is known that if $X \in \mathcal{C}_p$ and $Y$ is isomorphic to a subspace of a quotient space of $X$, then $Y \in \mathcal{C}_p$ \cite{JZ}. In particular, $\mathcal{C}_p$ contains every infinite-dimensional subspace of a quotient space of $\ell_p$.

\begin{cor} \lb{cor: subspaceofaquotient} 
Let $1<p<\infty$, $p\ne2$. If $X \in \mathcal{C}_p$ then $X$ does not admit an equivalent almost transitive norm.
In particular, no infinite-dimensional subspace of a quotient space of $\ell_p$ admits an equivalent almost transitive norm. \end{cor}
\begin{proof} $X$ is isomorphic to a subspace of $Z_p :=(\sum_{n=1}^\infty \oplus \ell_\infty^n)_{p}$ since every finite-dimensional normed space is $2$-isomorphic to a subspace of $\ell_\infty^n$ provided $n$ is sufficiently large.  $Z_p$ is an Asymptotic $\ell_p$-space with respect its natural basis $(e_i)$. By Theorem~\ref{thm: Asymptoticlp}, $X$ does not admit an equivalent almost transitive norm. 
\end{proof}

Recall (see  \cite[Section 1.9.1]{MMT}) that the \textit{upper envelope} is the norm $r_X$ on $c_{00}$ given by 
$$ r_X((a_i)) := \sup \left \{\left\|\sum_{i=1}^n a_i b_i\right\| \colon n\ge1, (b_i)_{i=1}^n \in \{X, (e_i)\}_n \right\},$$
and the \textit{lower envelope} is the function $g_X$ on $c_{00}$ given by

$$ g_X((a_i)) := \inf \left \{\left\|\sum_{i=1}^n a_i b_i\right\| \colon n\ge1, (b_i)_{i=1}^n \in \{X, (e_i)\}_n \right\}.$$

\begin{cor} \lb{cor: upperenvelope}
Suppose that $X$ has a Schauder basis $(e_i)$ and contains a subspace $Y$ which is almost transitive. Then
$$g_X((a_i)) \le \left(\sum_{i=1}^\infty |a_i|^{q_Y}\right)^{1/q_Y} \le \left(\sum_{i=1}^\infty |a_i|^{p_Y}\right)^{1/p_Y} \le r_X((a_i)),$$
with the obvious modification if $q_Y = \infty$. 
\end{cor}

\begin{theorem} \lb{upper2}
Suppose that $X$ has an unconditional basis $(e_i)$.  If $X$ contains an 
almost transitive subspace $Y$, then there exist $p \in [p_X, p_Y]$  and $q \in[q_Y,q_X]$  such that,  for $r=p$ and $r=q$ and for all $n \ge 1$ and $\varepsilon>0$, there exist disjointly  supported vectors $(x_i)_{i=1}^n\subset X$
 such that $(x_i)_{i=1}^n$ is  $(1+\varepsilon)$-equivalent to the unit vector basis of $\ell_r^n$. In particular, if $Y = X$, then $p = p_X$ and $q = q_X$. 
\end{theorem} 

\begin{proof} 
Sari \cite{S04} defined $\{X,(e_i)\}^d$ as the collection of all finite sequences $(w_i)_{i=1}^n$ such that, for some $m \ge n$, there exist $(b_i)_{i=1}^m \in \{X,(e_i)\}_m$ and a partition $\{A_1,\dots,A_n\}$ of $\{1,\dots,m\}$ such that $w_i = \sum_{j \in A_i} \alpha_j b_j$ for some scalars $(\alpha_j)_{j=1}^m$ and $\|w_i\|=1$. Thus, $(w_i)_{i=1}^n$ is a normalized basis for a block subspace of the asymptotic space with basis $(b_i)$. 

Recall (see  \cite[Definition 3.1]{S04}) that the  \textit{upper disjoint-envelope function}
is the norm $r^d_X$ on $c_{00}$ given by
$$ r^d_X((a_i)) := \sup  \left\{\left\|\sum_{i=1}^n a_i w_i\right\| \colon n\ge1, (w_i)_{i=1}^n \in \{X, (e_i)\}^d\right\}.$$ The \textit{lower disjoint-envelope function}  $g^d_X$ is defined similarly with  supremum replaced by  infimum. 
 Corollary~\ref{cor: upperenvelope} gives
$$g^d_X((a_i)) \le \left(\sum_{i=1}^\infty |a_i|^{q_Y}\right)^{1/q_Y} \le \left(\sum_{i=1}^\infty |a_i|^{p_Y}\right)^{1/p_Y} \le r^d_X((a_i)).$$
In particular, $r^d_X$ has \textit{power type} $p$ and $g^d_X$ has power type $q$  
 for some $p \in [1,p_Y]$ and $q \in [q_Y,\infty]$ (see \cite[Definition 5.3]{S04}). It follows from  \cite[Theorem 5.6]{S04} that $\{X, (e_i)\}^d$
contains the unit vector basis of $\ell_p^n$ and of $\ell_q^n$ for all $n\ge1$, which implies, for $r=p$ and $r=q$, the existence of disjointly supported vectors $(x_i)_{i=1}^n\subset X$
 such that $(x_i)_{i=1}^n$ is  $(1+\varepsilon)$-equivalent to the unit vector basis of $\ell_r^n$.
%
\end{proof}

\begin{rem} Theorem~\ref{upper2} holds also (with the same proof) under the weaker assumption that $X$ has asymptotic unconditional
structure. See \cite[Section~2.2.1]{MMT} for the definition of this notion. \end{rem} 

Recall that a basis satisfies $(p,q)$-estimates, where $1 < q \le p < \infty$, if there exists $C>0$ such that
\begin{equation*} \frac{1}{C}\left(\sum_{k=1}^n \|x_k\|^p\right)^{1/p} \le \left\|\sum_{k=1}^n x_k\right\|  \le C\left(\sum_{k=1}^n \|x_k\|^q\right)^{1/q},
 \end{equation*} 
whenever $x_1<x_2<\dots<x_n$. Theorem~\ref{prop: blockbasis} has the following immediate consequence.

\begin{cor} \label{cor: pqestimates}
Suppose that a Banach space  $X$ with a Schauder basis $(e_i)$ contains a subspace $Y$ which admits an equivalent almost transitive norm. 
 If   $(e_i)$ satisfies $(p,q)$-estimates, then $q \le 2 \le p$. 
\end{cor}

\begin{proof}  Let $|\!|\!|\cdot|\!|\!|$ be the equivalent almost transitive norm on $Y$.  Then (as in Theorem~\ref{lpandtsirelson})  $|\!|\!|\cdot|\!|\!|$ extends to an equivalent norm
$|\!|\!|\cdot|\!|\!|$ on $X$. Clearly, $(e_i)$ satisfies $(p,q)$-estimates under $|\!|\!|\cdot|\!|\!|$. Hence \eqref{eq: blockbasis} gives the desired conclusion. 
\end{proof}

A natural question, in the light of Theorem~\ref{thm: Asymptoticlp}, is whether every super-reflexive space which does not admit an almost transitive norm must contain an asymptotic-$\ell_p$ space? The next result answers this question negatively. 

\begin{cor} \lb{example}
There exist super-reflexive spaces which do not contain either an Asymptotic-$\ell_p$  space or  a subspace  which admits an 
almost transitive norm. \end{cor} 

\begin{proof} The spaces $S_{q,r}(\log_2(x+1))$  ($1 < q<r<\infty$) constructed in \cite{CKKM} are super-reflexive and satisfy $(r,q)$-estimates. 
Hence, if $1<q<r<2$ or $2<q<r<\infty$,  $S_{q,r}(\log_2(x+1))$ does not contain any subspace which admits an almost transitive norm. Moreover,  $S_{q,r}(\log_2(x+1))$ is complementably minimal, has a subsymmetric basis, and does not contain a copy of any $\ell_p$ space \cite{CKKM}, which is easily seen to preclude the containment of an Asymptotic-$\ell_p$ space. \end{proof} 

\subsection{Orlicz spaces}

 \begin{theorem} \lb{orlicz}
Suppose that $X$ is an Orlicz sequence space $\ell_M$ (where $M$ is an Orlicz function). Then $X$ contains a subspace $Y$ which admits an almost transitive norm if and only if
$X$ contains a subspace isomorphic to $\ell_2$.
\end{theorem}

 \begin{proof} Corollary~\ref{cor: pqestimates} implies that the Matuszewska-Orlicz indices of $M$ satisfy $\alpha_M \le 2 \le \beta_M$,
which in turn implies by a theorem of Lindenstrauss and Tzafriri \cite{LT71, LT73} (see also \cite{LT})  that $\ell_M$ contains a subspace isomorphic to $\ell_2$.
 Conversely, if $X$ contains a subspace $Y$ isomorphic to $\ell_2$ then $Y$ admits an equivalent transitive norm.
\end{proof}

\subsection{Convex transitive spaces}

We say that a Banach space $X$ is  {\it convex transitive} if for any $x$ in the unit sphere of $X$, $\overline{\conv}\{Tx: T\in\Isom(X,\|\cdot\|)\}$ is equal to the unit ball of $X$. This notion was introduced in \cite{PR62} where it was shown that in general it is weaker than almost transitivity. However 
  in super-reflexive Banach spaces convex transitivity is equivalent to almost transitivity \cite{CS98}.   A long list of additional related results is summarized in
\cite[Theorem 6.8 and Corollary~6.9]{BGRP2002}, see also \cite{T08,FR2011}. Thus, in particular, we obtain from Corollary~\ref{cor: subspaceofaquotient}:

\begin{cor} \lb{cor: convextransitive} 
For $1<p<\infty$, $p\ne 2$,  no infinite-dimensional subspace of a quotient space of $\ell_p$ admits a convex transitive renorming. 
\end{cor}

It is well known that the spaces $L_p[0,1]$, $1<p<\infty$, with the original norm are almost transitive. Next we consider their subspaces which admit an almost transitive renorming. 

\begin{cor} \lb{subspaceofLp}
Let $X$ be a subspace of $L_p[0,1]$, $2<p<\infty$, which admits an equivalent convex transitive norm. Then $X$ contains a subspace isomorphic to
$\ell_2$. 
\end{cor}  
\begin{proof} By \cite{JO} either $X$  contains a subspace isomorphic to $\ell_2$ or $X$ is isomorphic to a subspace of $\ell_p$. By Corollary~\ref{cor: convextransitive}, the latter case would contradict the fact that $X$ admits a convex transitive norm, which proves the corollary.
\end{proof} 

\begin{cor} \lb{pin2infty}
Let $X$ be a subspace of $L_p[0,1]$, $2<p<\infty$, or, more generally, of any non-commutative $L_p$-space  \cite{HJX10} for  $2<p<\infty$, so that every subspace $Y$ of $X$ admits an equivalent convex transitive norm, then $X$ is isomorphic to $\ell_2$.
\end{cor} 
\begin{proof}
In the commutative case by \cite{KP}, and in the non-commutative case  by \cite[Theorem~0.2]{RX03}, either $X$ is isomorphic to $\ell_2$ or $X$ contains a subspace $Y$ isomorphic to $\ell_p$.
By Corollary~\ref{cor: convextransitive}, in the latter case $Y$ does not admit an equivalent convex transitive norm, which proves the corollary.
\end{proof} 

\begin{cor} \lb{schatten}
Let $X$ be a subspace of the Schatten class $S^p(\ell_2)$, $1<p<\infty$, $p\ne 2$, so that every subspace $Y$ of $X$ admits an equivalent convex transitive norm, then $X$ is isomorphic to $\ell_2$.
\end{cor} 
\begin{proof}
The proof is the same as that of Corollary~\ref{pin2infty}, except that we use 
 \cite{F75} to conclude that either $X$ is isomorphic to $\ell_2$ or $X$ contains a subspace $Y$ isomorphic to $\ell_p$, which, by Corollary~\ref{cor: convextransitive}, gives the conclusion of the corollary.
\end{proof} 

\begin{cor} \lb{pin12}
Let $X$ be a subspace of $L_p[0,1]$, $1 < p<2$, such that every subspace $Y$ of $X$ admits an equivalent convex transitive norm.
Then, for all $1\le r<2$, $X$ is isomorphic to a subspace of $L_r[0,1]$ and
 every subspace of X contains almost isometric copies of $\ell_2$.
\end{cor} 
\begin{proof} By Corollary~\ref{cor: convextransitive}, $X$ does not contain a copy of $\ell_r$ for any $1\le r<2$. Thus, by  a theorem of Rosenthal \cite{R73}, $X$ is contained in $L_r[0,1]$ for all $p\le r<2$. The latter implies by a theorem of Aldous \cite{A} that every subspace of $X$ contains isomorphic (even almost isometric \cite{KM}) copies of $\ell_2$.
 \end{proof} 

\begin{remark}\lb{stable}
The last statement of Corollary~\ref{pin12} generalizes to the class of \textit{stable} spaces introduced in \cite{KM}.  If $X$ is stable then every subspace of $X$ contains almost isometric copies of $\ell_p$ for some $1 \le p < \infty$ \cite{KM}. Thus, if, in addition, every subspace of $X$ admits an equivalent convex transitive norm, then by Theorem~\ref{lpandtsirelson}
every subspace of $X$ contains almost isometric copies of $\ell_2$. \end{remark}

\begin{remark}\lb{smooth} 
If $X$  is stable,  
 $X$ admits a $C^2$-smooth  bump, and every subspace $Y$ of $X$ admits an equivalent convex transitive norm, then $X$ is isomorphic to a Hilbert space. 

This follows from
the last statement of Corollary~\ref{pin12} and Remark~\ref{stable}, since by \cite[Corollary V.5.2]{DGZ} if $X$ is an infinite dimensional Banach space $X$ which admits a $C^2$-smooth  bump and is saturated with subspaces of cotype 2,  then $X$ is isomorphic to a Hilbert space. 
\end{remark}

\begin{remark}\lb{noncom}
We do not know whether  Corollary~\ref{pin12}  generalizes to the setting of non-commutative $L_p$-spaces. 
 Rosenthal's theorem used in the proof of Corollary~\ref{pin12}
generalizes to non-commutative $L_p$-spaces,  \cite{JP08} and \cite{R08}. However  non-commutative $L_p$-spaces  are not stable in general \cite{M97}. As far as we know,  it is not known whether Aldous's theorem can be generalized to the non-commutative setting.
\end{remark}

\section{Maximal bounded subgroups of the isomorphism group}

As noted in the Introduction, maximal isometry groups of equivalent renormings of a Banach space $X$ are exactly maximal bounded subgroups of the group $\GL(X)$ of isomorphisms from $X$ onto $X$. Thus all results in this and the next section can be stated equivalently in the terminology of  bounded subgroups of  $\GL(X)$. We choose the terminology of isometry groups of renormings of $X$ since our arguments rely heavily on  Rosenthal's characterization of isometry groups of a general class of Banach spaces, which we recall below.

We will need the following definitions and results from \cite{R86}.

 A Banach space $X$ with a normalized 1-unconditional basis $\{e_\g\}_{\g\in \G}$ is called {\it impure} if there exist $\al\ne\be$ in $\G$ so that $(e_\al,e_\be)$ is isometrically equivalent to the usual basis of 2-dimensional $\ell_2^2$ and for all $x, x'\in \Span(e_\al,e_\be)$ with $\|x\|=\|x'\|$ and for all $y\in\Span\{e_\g:\g\ne\al,\be\}$ we have  $\|x+y\|=\|x'+y\|$. Otherwise the space $X$ is called {\it pure} (cf. \cite[Corollary~3.4]{R86}). For convenience, we will also say that $\{e_\g\}_{\g\in \G}$ is pure (resp.\ impure) if $(X, \{e_\g\}_{\g\in \G})$ is pure (resp.\ impure).

\begin{defn} (\cite[p. 430 and Proposition~1.11]{R86}) Let $X$ be a Banach space and $Y$ be a subspace of $X$. $Y$ is said to be {\it well-embedded in $X$} if there exists a subspace $Z$ of $X$ so that $X=Y+Z$ and for all $y, y'\in Y, z\in Z$, if $\|y\|=\|y'\|$ then $\|y+z\|=\|y'+z\|$.

$Y$ is called a {\it well-embedded Hilbert space} if $Y$ is well-embedded and Euclidean.
$Y$ is called a {\it Hilbert component of X} if $Y$ is a maximal well-embedded Hilbert subspace (a similar concept in complex Banach spaces was introduced by Kalton and Wood \cite{KW76}).

If $X$  is space with a 1-unconditional basis  $E=\{e_\g\}_{\g\in \G}$ and $(H_\g)_{\g\in \G}$ are Hilbert spaces all of dimension at least 2, then $Z=(\sum_\G\oplus H_\g)_E$ is called a \textit{functional hilbertian sum} \cite{R86}. 
\end{defn}

Rosenthal \cite{R86} proved the following useful fact:
\begin{theorem}  (\cite[Theorem~1.12]{R86})\label{comp}
If $Y$ is a well-embedded Hilbert sub\-space of $X$, then
there exists a Hilbert component of $X$ containing $Y$.
\end{theorem}

 The main result of  \cite{R86} which we will use extensively is the following:

\begin{theorem} (\cite[Theorem~3.12]{R86})\label{FHS}
Let $X$  be a pure space with a 1-unconditional basis  $E=\{e_\g\}_{\g\in \G}$ and $(H_\g)_{\g\in \G}$ be Hilbert spaces all of dimension at least 2, and let $Z=(\sum_\G\oplus H_\g)_E$ be the corresponding functional hilbertian sum. Let $P(Z)$ denote the set of all bijections $\s:\G\to\G$ so that
\begin{enumerate}
\item[(a)] $\{e_{\s(\g)}\}_{\g\in \G}$ is isometrically equivalent to $\{e_\g\}_{\g\in \G}$, and
\item[(b)] $H_{\s(\g)}$ is isometric to $H_\g$ for all $\g\in \G$.
\end{enumerate}

Then $T:Z\to Z$ is a surjective isometry if and only if there exist $\s\in P(Z)$ and surjective linear isometries $T_\g: H_\g\to H_{\s(\g)}$, for all $\g\in \G$,  so that for all $z=(z_\g)_{\g\in \G}$ in $Z$, and for all $\g\in \G$, 
\begin{equation}\lb{isofiber}
(Tz)_{\s(\g)}=T_\g(z_\g).
\end{equation}  
In particular, if $T\in\Isom(Z)$ and $H$ is a Hilbert component of $Z$, then $T(H)$ is a Hilbert component of $Z$.
\end{theorem}

Theorem~\ref{FHS} is valid for both real and complex spaces. For separable complex Banach spaces it was proved earlier  by Fleming and Jamison \cite{FJ74}, cf. also \cite{KW76}.

As a consequence of \cite[Theorem~2]{R86} we obtain a condition when maximal isometry groups of
functional  hilbertian sums  are conjugate to each other.

\begin{prop} \label{prop: conjugate}
 Suppose $(Z, \|\cdot\|)$ has two renormings $\|\cdot\|_1$ and $\|\cdot\|_2$ such that $(Z, \|\cdot\|_1)$ is isometric to a functional hilbertian sum, $Z_1 = (\sum_{\Gamma_1} \oplus H_{\gamma})_{E_1}$, and $(Z, \|\cdot\|_2)$ is isometric to a functional hilbertian sum, $Z_2 = (\sum_{\Gamma_2} \oplus H_{\gamma})_{E_2}$, where $E_1$ and $E_2$ are pure.  Suppose $G_1 := \operatorname{Isom}(Z, \|\cdot\|_1)$  and $G_2 := \operatorname{Isom}(Z, \|\cdot\|_2)$ are conjugate in the isomorphism group of $(Z, \| \cdot \|)$ and are maximal. Then there exists a bijection $\rho \colon \Gamma_1 \rightarrow \Gamma_2$ such that $H_\gamma$ is isometric to $H_{\rho(\gamma)}$ for all $\gamma \in \Gamma_1$.\end{prop}
\begin{proof}  Let $G_1 = T^{-1}G_2T$ for some isomorphism $T$ of $(Z, \|\cdot\|)$. Define 
$$ \|z\|_3 := \sup_{g \in G_1} \|g(z)\| \qquad (z \in Z).$$
Clearly, $\|\cdot\|_3$ is $G_1$-invariant and is equivalent to $\|\cdot\|$. Since $(Z,\|\cdot\|_1)$ is isometric to the functional hilbertian sum $Z_1$, and since $G_1$ is its isometry group, it follows that $(Z, \| \cdot \|_3)$ is isometric to a functional hilbertian sum $Y_1 = (\sum_{\Gamma_1} \oplus H_{\gamma})_{F_1}$.
Moreover, by maximality of $G_1$, we have that $G_1 =  \operatorname{Isom}(Z, \|\cdot\|_3)$. In particular, since $E_1$ is pure, it follows that $F_1$ is also pure, for otherwise the isometry group of $(Z, \| \cdot \|_3)$ would strictly contain $G_1$. On the other hand,
$$\|z\|_3 = \sup_{g \in G_2} \|T^{-1}gT(z)\|,$$ 
which implies likewise that $(Z, \|\cdot\|_3)$ is isometric to a functional Hilbert sum $Y_2 = (\sum_{\Gamma_2} \oplus H_{\gamma})_{F_2}$, where $F_2$ is pure.
The existence of the bijection $\rho$ now follows from a uniqueness theorem of Rosenthal for pure functional  hilbertian sums \cite[Theorem~2]{R86}.
\end{proof}

We are now ready to describe a countable number of  different equivalent maximal norms on  Banach spaces with 1-symmetric bases, which are not isomorphic to $\ell_2$. Henceforth we shall say that a basis $E$ for a Banach space $X$ is \textit{non-hilbertian} if $X$ is not isomorphic to a Hilbert space.

\begin{theorem} \lb{nmax}
Let $X$ be a pure Banach space with a non-hilbertian  1-symmetric  basis $E=\{e_k\}_{k=1}^\infty$, 
let $n\in\bbN, n\ge 2$, and $Z_n=Z_n(X)=(\sum_{k=1}^\infty \oplus H_k)_E$, where, for all $k\in\bbN$, $H_k$ is isometric to $\ell_2^n$.
 Then  $Z_n$ is  isomorphic to $X$ and the isometry group of $Z_n$ is maximal. 

Moreover,  if $n \ne m$ then $\Isom(Z_n)$  and $\Isom(Z_m)$ are not conjugate to each other in the isomorphism group of $X$.
\end{theorem}

\begin{proof}
It is easily seen that $Z_n$ is isomorphic to the direct sum of $n$ copies of $X$ and hence isomorphic to $X$ itself since $X$  has a symmetric basis.

By Theorem~\ref{FHS}, all isometries of $Z_n$ have form \eqref{isofiber}, and, since the basis is 1-symmetric and all $H_k$ are isometric to each other, the set $P(Z_n)$ is equal to the set of all bijections of $\bbN$.

Suppose that $Z_n$ has a renorming $\widetilde{Z}_n=(Z_n, |\!|\!|\cdot|\!|\!|)$ so that $\Isom(\widetilde{Z}_n)\supseteq\Isom(Z_n)$. Then the 1-unconditional basis of $Z_n$ is also 1-unconditional in $\widetilde{Z}_n$, and for all $k\in \bbN$ the subspace $H_k$ is well-complemented in $\widetilde{Z}_n$. By Theorem~\ref{comp}, for each $k\in\bbN$, there exists a Hilbert component of $\widetilde{Z}_n$ containing $H_k$. Thus every  Hilbert component of $\widetilde{Z}_n$ has dimension greater than or equal to $n$, and $\bbN$ can be split into disjoint subsets $\{A_j\}_{j\in J}$, so that every  Hilbert component of $\widetilde{Z}_n$ is given by
$$\widetilde{H}_j=\left (\sum_{k\in A_j} \oplus H_k\right)_2,$$
and
$$\widetilde{Z}_n= \left(\sum_{j\in J} \oplus \widetilde{H}_j\right)_{\{\widetilde{e}_j\}_{j\in J}},$$
for some 1-unconditional basis $\{\widetilde{e}_j\}_{j\in J}$. Since $Z_n$, and thus also $\widetilde{Z}_n$, is not isomorphic to $\ell_2$, $J$ is  not finite.

Therefore, by Theorem~\ref{FHS}, all isometries of $\widetilde{Z}_n$ have form \eqref{isofiber}.

Suppose that there exists $k_0\in\bbN$ so that $H_{k_0}$ is strictly contained in a component $\widetilde{H}_{j_0}$ of $\widetilde{Z}_n$. Since $\widetilde{H}_{j_0}$ is strictly larger than $H_{k_0}$, there exists $k_0'\ne k_0$ so that $H_{k'_0}\subset \widetilde{H}_{j_0}$. Let $ j_1, j_2\in J$ be  distinct indices in $J$, both different from $j_0$, and let $k_1, k_2 \in\bbN$ be such that $H_{k_1}\subset \widetilde{H}_{j_1}$ and $H_{k_2}\subset \widetilde{H}_{j_2}$.
Let $\s:\bbN\to\bbN$ be bijection  such that $\s(k_0)=k_1$ and $\s(k'_0)=k_2$. Let $T:Z_n\to Z_n$ be defined for all $z=(z_k)_{k\in\bbN}\in Z_n$ by
$$(Tz)_{\s(k)}=z_k.$$
 
By Theorem~\ref{FHS}, $T\in \Isom(Z_n)$, and thus, by assumption $T\in \Isom(\widetilde{Z}_n)$. However $T(\widetilde{H}_{j_0})\cap \widetilde{H}_{j_1}\ne \emptyset$ and $T(\widetilde{H}_{j_0})\cap \widetilde{H}_{j_2}\ne \emptyset$. Thus  $T(\widetilde{H}_{j_0})$ is not a component of $\widetilde{Z}_n$, which contradicts the fact that $T\in \Isom(\widetilde{Z}_n)$.
Hence every Hilbert component of $Z_n$ is a Hilbert component of $\widetilde{Z}_n$ and,  since $P(Z_n)$ contained all bijections of $\bbN$, $P(\widetilde{Z}_n)\subseteq P(Z_n)$. Hence $\Isom(\widetilde{Z}_n)\subseteq \Isom(Z_n)$, and thus $\Isom(Z_n)$ is maximal.

The `moreover' sentence follows from Proposition~\ref{prop: conjugate}.\end{proof}

\begin{rem}\label{ST2max} 
Theorem~\ref{nmax} applies in particular to 
the space $S(T^{(2)})$, the  symmetrization of the 2-con\-ve\-xified Tsirelson space, see \cite{CS89}. Indeed,  it is known that $S(T^{(2)})$ does not contain $\ell_2$, and
it is easy to verify that for all $k,l \in \bbN$, $\|e_k+e_l\|_{S(T^{(2)})}=1$, and thus the standard basis of $S(T^{(2)})$  is pure.  It is clear that the isometry groups of renormings described in Theorem~\ref{nmax} are not almost transitive. 

It is known that any symmetric weak Hilbert space is Hilbertian, but in some sense the space $S(T^{(2)})$ is very close to a weak Hilbert space, see \cite[Note A.e.3 and Proposition A.b.10]{CS89}.

We do not know whether or not the space $S(T^{(2)})$ admits an almost transitive renorming, or whether there exists any   non-hilbertian symmetric space which admits an almost transitive renorming.
\end{rem}

\begin{theorem}\lb{continuum} Let $X$ be a pure Banach space with a non-hilbertian  1-symmetric  basis $E=\{e_k\}_{k=1}^\infty$ and let $Z = (\sum_{k=1}^\infty \oplus \ell_2^k)_E$.
Let $J$ be any subset of $\bbN$, with $\min J \ge 2$, and let $\bbN=\bigcup_{j\in J} A_j$, where $A_j$ are disjoint infinite subsets of $\bbN$. For $k\in A_j$, let $H_k=\ell_2^j$, and let 
$$Z_J=Z_J(E)=\left(\sum_{k=1}^\infty \oplus H_k\right)_{E}.$$

Then $Z_J$ is isomorphic to $Z$, $\Isom(Z_J)$ is maximal, and, if $J\ne J'$ then $\Isom(Z_J)$ and $\Isom(Z_{J'})$  are not conjugate in the group of isomorphisms of $Z$.
Hence there are continuum different (pairwise non-conjugate) maximal  isometry groups  of renormings of $Z$.
\end{theorem}

\begin{proof} It is well-known (see \cite{BCLT}) that  if $(k_j)_{j=1}^\infty$ is any \textit{unbounded} sequence of positive integers then $(\sum_{j=1}^\infty \oplus \ell_2^{k_j})_E$ is $4$-isomorphic to $Z$. In particular, $Z_J$ is 
$4$-isomorphic to $Z$ for all $J$.
%

Let $P(Z_J)$ be the set defined in Theorem~\ref{FHS}. Then, by the symmetry of $E$, $P(Z_J)$ consists of all bijections $\s:\bbN\to\bbN$, so that, for all $j\in J$, $\s(A_j)=A_j$.

By Theorem~\ref{FHS}, $\Isom(Z_J)$ consists of all maps $T:Z_J\to Z_J$ so that 
 there exist $\s\in P(Z_J)$ and surjective linear isometries $T_k: H_k\to H_{\s(k)}$, for all $k\in \bbN$,  so that for all $z=(z_k)_{k\in \bbN}\in Z_J$, where $z_k\in H_k$, and for all $k\in \bbN$, 
\begin{equation}\lb{isofibercont}
(Tz)_{\s(k)}=T_k(z_k).
\end{equation}  
As in Theorem~\ref{nmax}, let $Z_j = (\sum_{k \in A_j}\oplus \ell_2^j)_E$. Then $T\in \Isom(Z_J)$ \wtw   there exist surjective linear isometries $S_j:Z_j\to Z_j$, $j\in J$, so that for all  $z=(\widetilde{z_j})_{j\in J}\in Z_J$, where $\widetilde{z_j}\in Z_j$, and for all $j\in J$, 
\begin{equation}\lb{isofibercont2}
\widetilde{(Tz)_{j}}=S_j(\widetilde{z_j}).
\end{equation}  
Thus 
$$ \Isom(Z_J)= \sum_{j\in J} \oplus \Isom(Z_j).$$


The proof that $\Isom(Z_J)$ is maximal is essentially the same as the proof that $\Isom(Z_n)$ is maximal in Theorem~\ref{nmax}.

The fact that if $J\ne J'$ then  $\Isom(Z_{J})$ and  $\Isom(Z_{J'})$  are not conjugate to each other in the isomorphism group of $Z$ follows from Proposition~\ref{prop: conjugate}. 
\end{proof}

\begin{rem} As $Z$ is a separable Banach space  the collection of equivalent norms on $Z$ has cardinality $\frak{c}$. Hence Theorem~\ref{continuum} implies that the  cardinality of any maximal  collection of pairwise non-conjugate maximal bounded subgroups of $\GL(Z)$ 
  is exactly equal to $\frak{c}$. \end{rem}  
Let $E_p$ be the standard unit vector basis of $\ell_p$. It is a well-known consequence of the fact that $\ell_p$ is a prime space \cite{P60}
that $\ell_p$ is isomorphic to $Z_p = (\sum_{k=1}^\infty \oplus \ell_2^k)_{E_p}$ for $1<p<\infty$. Hence we obtain the following consequence
of Theorem~\ref{continuum}. 

\begin{theorem} For $1<p <\infty$, $p \ne 2$, $\ell_p$ admits a continuum of  renormings whose isometry groups are maximal and are not pairwise conjugate
in the isomorphism group of $\ell_p$. \end{theorem} 

The latter theorem may be generalized as follows. 

\begin{prop}\lb{gencontinuum} Let $X$ be a space with a non-hilbertian symmetric basis $E$ such that
\begin{itemize} 
\item $X(X)$  (i.e., the $E$-sum of infinitely many copies of $X$) is isomorphic to $X$, and
\item $X$ contains uniformly complemented and uniformly isomorphic copies of $\ell_2^n$.
\end{itemize} Then $X$ is isomorphic to $Z = (\sum_{k=1}^\infty \oplus \ell_2^k)_E $ and hence $X$ has a continuum of renormings whose isometry groups are maximal and are not
pairwise conjugate in the isomorphism group of $X$. \end{prop}

\begin{proof} The hypotheses imply that $Z$ is isomorphic to a complemented subspace of $X$. Since $Z$ is isomorphic to its square
$Z \oplus Z$ (e.g., with the sum norm) it follows from the Pe\l czy\'nski decomposition method (see e.g. \cite{LT}) that $Z$ is isomorphic to $X$.
\end{proof}

\begin{rem}
We note that there exist symmetric  spaces $X$ which are not isomorphic to $\ell_p$ and so that $X(X)$ is isomorphic to $X$, see \cite{Read}.
\end{rem}

\section{Isometry groups not contained in any maximal bounded subgroup of the isomorphism group}\lb{secnonmax}

Let $S$ be the collection of all partitions $\mathcal{B}=(B_k)_{k=1}^\infty$ of $\bbN$ into finite sets of bounded size, i.e. 
$$N_{\mathcal{B}} :=\max_k |B_k|<\infty.$$ We define a partial order  $\le$ on $S$ by 
$\mathcal{B}\le \tilde{\mathcal{B}}$ if $\mathcal{B}$ is a refinement of $\tilde{\mathcal{B}}$, and let $S_{\mathcal B}=\{\tilde{\mathcal{B}}\in S\ : \ \mathcal{B}\le \tilde{\mathcal{B}}\}$.

 The following properties are easily proved:
\begin{itemize}
\item $(S,\le)$ has cardinality of the continuum,
\item $(S,\le)$ contains order-isomorphic copies of every countable ordinal,
\item $(S_{\mathcal B},\le)$ is order-isomorphic to $(S,\le)$ with $\mathcal{B}\longleftrightarrow (\{k\})_{k=1}^\infty$ in this isomorphism.
\end{itemize}

Let $E_p=\{e_k\}_{k=1}^\infty$ be the standard basis for $\ell_p$, where $1<p<\infty$, $p\ne 2$, and, for $k\ge 1$, let $H_k$ be isometric to $\ell_2^{2^k}$.  We define a space $Y$ as follows:
\begin{equation}\lb{defy}
Y:= \left(\sum_{k=1}^\infty\oplus H_k\right)_{E_p}.
\end{equation} 

By \cite{P60} $Y$ is isometric to a renorming of $\ell_p$.
Let $(x_i)_{i=1}^\infty \subset \ell_p$ be the basis of $\ell_p$ which is sent to the natural basis of $Y$ under the isometry.

For  each $\mathcal{B}\in S$, consider the \textit{specific} renorming of $\ell_p$  given by 
$$Y_\mathcal{B}:= \left(\sum_{k=1}^\infty\oplus H^\mathcal{B}_k\right)_{E_p},$$ where $ H^\mathcal{B}_k= (\sum_{i\in B_k}\oplus H_i)_{\ell_2},$ 
for which $(x_i)$ is the basis of $\ell_p$ corresponding to the natural basis of $Y_\mathcal{B}$. This is possible because the condition
 $N_{\mathcal{B}} :=\max_k |B_k|<\infty$ guarantees  
that the norm of $Y_\mathcal{B}$ is equivalent to the norm of $Y$. In fact, we have 
$$ \|x\|_Y \le \|x\|_{Y_\mathcal{B}} \le N_{\mathcal{B}}^{1/p - 1/2} \|x\|_Y \qquad(x \in \ell_p, 1<p<2)$$
and
$$ N_{\mathcal{B}}^{1/p - 1/2}\|x\|_Y \le \|x\|_{Y_\mathcal{B}} \le \|x\|_Y \qquad(x \in \ell_p, 2<p<\infty).$$
We identify $Y_\mathcal{B}$ with this particular renorming of $\ell_p$ in which the basis $(x_i)$ of $\ell_p$ corresponds to the natural basis of $Y_\mathcal{B}$. Note that  $Y_\mathcal{B}$ is not a maximal renorming of $\ell_p$ since
$$  \Isom(Y_\mathcal{B}) \subsetneq \Isom(Y_{\tilde{\mathcal{B}}})$$
for all $\tilde{\mathcal{B}} \in S_\mathcal{B} \setminus \{\mathcal{B}\}$.

\begin{prop}\lb{conjugate}
Let $\mathcal{B}, \tilde{\mathcal{B}}\in S$. If $\Isom(Y_\mathcal{B})$ and $\Isom(Y_{\tilde{\mathcal{B}}})$ are conjugate in the isomorphism group of $\ell_p$ then $\mathcal{B}=\tilde{\mathcal{B}}$.
\end{prop}

\begin{proof} This does not follow from Proposition~\ref{prop: conjugate} because the maximality hypothesis is not satisfied.
Note that $n_k^\mathcal{B}:=\dim H^\mathcal{B}_k=\sum_{i\in B_k} 2^i$. By the uniqueness of binary representations  the map $k\mapsto n_k^\mathcal{B}$ is one-to-one.

Since $\Isom(Y_\mathcal{B})$ and
 $\Isom(Y_{\tilde{\mathcal{B}}})$
 are conjugate to each other  it follows that the  $\Isom(Y_\mathcal{B})$-invariant
 and   $\Isom(Y_{\tilde{\mathcal{B}}})$-invariant subspaces of $\ell_p$ have the same dimensions. 
By Theorem~\ref{FHS} there is an $n$-dimensional $\Isom(Y_\mathcal{B})$-invariant subspace of $\ell_p$ if and only if $n = \sum_{k \in A}n_k^\mathcal{B}$ for some finite $A \subset \mathbb{N}$, and similarly for  $\Isom(Y_{\tilde{\mathcal{B}}})$. Thus, by uniqueness of the binary representation, $\mathcal{B}=\tilde{\mathcal{B}}$.
\end{proof}

\begin{prop}\lb{Brenormings}
Let     
$\mathcal{B}\in S$ and let $(\ell_p, \|\cdot\|_0)$ be a renorming of $\ell_p$ so that
$\Isom(\ell_p, \|\cdot\|_0)\supseteq\Isom(Y_\mathcal{B})$. Then there exists  $\tilde{\mathcal{B}}\in S_\mathcal{B}$ so that
$$\Isom(\ell_p, \|\cdot\|_0)=\Isom(Y_{\tilde{\mathcal{B}}}).$$ In particular, $\Isom(\ell_p, \|\cdot\|_0)$ is not maximal.
Conversely, every $\tilde{\mathcal{B}}\in S_\mathcal{B}$ determines such a renorming. 
\end{prop}

\begin{proof}
By Theorem~\ref{FHS}, it is clear that for every $\tilde{\mathcal{B}}\in S_\mathcal{B}$, $\Isom(Y_{\tilde{\mathcal{B}}})\supseteq\Isom(Y_\mathcal{B})$, which is the last sentence of the theorem.

Now suppose that $\Isom(\ell_p, \|\cdot\|_0)\supseteq\Isom(Y_\mathcal{B})$. Thus, $(\ell_p, \|\cdot\|_0)$ is a functional hilbertian sum
$$(\ell_p,\|\cdot\|_0)=\left(\sum_{k=1}^\infty\oplus H^\mathcal{B}_k\right)_{E},$$
for some (possibly impure) $1$-unconditional basis $E$. Moreover, $(x_i)_{i=1}^\infty$ is the basis of $\ell_p$ corresponding to the natural basis of $(\sum_{k=1}^\infty\oplus H^\mathcal{B}_k)_{E}$.
Arguing as in the proof of Theorem~\ref{nmax},
  $\bbN$ can be partitioned into disjoint subsets $\{A_j\}_{j\in J}$, so that every  Hilbert component of $(\ell_p, \|\cdot\|_0)$ is given by
$$\widetilde{H}_j= \left(\sum_{k\in A_j} \oplus H^\mathcal{B}_k\right)_2,$$
and
$$(\ell_p, \|\cdot\|_0)= \left(\sum_{j\in J} \oplus \widetilde{H}_j\right)_{E_0},$$
for some  pure $1$-unconditional basis $E_0$. Let $$\tilde B_j = \cup_{k \in A_j} B_k \qquad(j\ge1).$$ Since $\|\cdot\|_0$ is equivalent to
$\|\cdot\|_{Y_\mathcal{B}}$ it follows that $\max_{j\ge1} |\tilde B_j| < \infty$.                Let $\tilde{\mathcal{B}} 
 = (\tilde B_j)_{j=1}^\infty$.
Then $\tilde{\mathcal{B}} \in S_\mathcal{B}$, and (since $E_0$ is pure) Theorem~\ref{FHS} gives
$$\Isom(\ell_p, \|\cdot\|_0)=\Isom(Y_{\tilde{\mathcal{B}}}).$$
Since (as observed above) $Y_{\mathcal{B}}$ is not maximal for any $\mathcal{B} \in S$, it follows that $(\ell_p, \|\cdot\|_0)$ is not maximal. 
 \end{proof}

Combining the last two propositions gives the main result of this section.

\begin{theorem}\lb{Bnonmax} 
For $1< p<\infty$, $p\ne2$, $\ell_p$  has a  continuum of  renormings  none of whose isometry groups  is  contained in any maximal bounded subgroup of the isomorphism group  of $\ell_p$. Moreover, these isometry groups are not pairwise conjugate in $\GL(\ell_p)$.
\end{theorem}

\begin{rem} \lb{gennonmax}
Symmetry of the standard basis of $\ell_p$ is not used in the proof of Theorem~\ref{Bnonmax}. In particular, the theorem holds for any space $Z$ 
with the following properties:
 \begin{itemize}
\item $Z$ has an unconditional basis $E$ such that no subsequence of $E$ is equivalent to the unit vector basis of $\ell_2$; 
\item $Z$ is isomorphic to  $(\sum_{n=1}^\infty\oplus H_n)_{E}$ 
for every collection $(H_n)_{n=1}^\infty$ of finite-dimensional Hilbert spaces. 
\end{itemize}
These properties are satisfied for example by the space $(\sum_{n=1}^\infty \oplus \ell_2^n)_{E}$, where $E$ is   symmetric,  pure, and non-hilbertian \cite{BCLT}.  
\end{rem}

\begin{prop} 
$T^{(2)}$ admits a continuum of renormings  none of whose isometry groups is contained in any   maximal bounded subgroup of the isomorphism group  of $T^{(2)}$. Moreover, these isometry groups are not pairwise conjugate in $\GL(T^{(2)})$. 
\end{prop}
\begin{proof} The standard basis $\{e_{i}\}_{i=1}^\infty$ of  $T^{(2)}$ is pure 
since $\|e_i+e_j\|_{T^{(2)}}=1$ for all $i\ne j$. For each $J \subseteq \{ 2n \colon n \ge 1\}$,
let $\{n^J_k\}_{k\ge1}$ be the arrangement of $J \cup \{2n-1 \colon n \ge 1\}$ as an increasing sequence, and let $m^J_k := 2^{n^J_k}$. 
Note that for all $J$ and $k \ge 1$, we have
 \begin{equation} \label {eq: growthcondition} m^J_k \le 2^{2k-1}. \end{equation}  
Let $A^J = \cup_{k\ge1} \{i  \colon m^J_k \le i < 2 m^J_k\}$. Then by \cite[Corollary 12]{CJT84} the growth condition \eqref{eq: growthcondition} ensures that
 the subsequences $\{e_i\}_{i \in A^J}$ and $\{e_{m^J_k}\}_{k\ge1}$ of the basis $\{e_i\}_{i\ge1}$ are in fact 
both equivalent to the whole sequence $\{e_i\}_{i\ge1}$.

Note that from the definition of the $T^{(2)}$ norm we have that $\{e_i \colon m^J_k \le i < 2 m^J_k\}$ is $2$-equivalent to the unit vector basis of $\ell_2^{m^J_k}$.
 Let $F^J_k := \Span(e_i \colon m^J_k \le i < 2 m^J_k)$. Suppose $x_k \in F^J_k$ ($k\ge1$). Then by \cite[Corollary 7]{CJT84},
$$\frac{1}{\sqrt{3}} \left\|\sum_{k=1}^\infty \|x_k\| e_{m^J_k}\right\|_{T^{(2)}} \le\left \| \sum_{k=1}^\infty x_k \right\|_{T^{(2)}} \le  \sqrt{18} \left\|\sum_{k=1}^\infty \|x_k\| e_{m^J_k}\right\|_{T^{(2)}}.$$  
Since $\|\cdot\|_2$ and $\|\cdot\|_{T^{(2)}}$ are $2$-equivalent on $F^J_k$ and since $\{e_{m^J_k}\}_{k\ge1}$ is equivalent to $\{e_i\}_{i\ge1}$ there exists a constant $C>0$ such that
$$\frac{1}{C} \left\|\sum_{k=1}^\infty \|x_k\|_2 e_{k}\right\|_{T^{(2)}} \le \left\| \sum_{k=1}^\infty x_k\right \|_{T^{(2)}} \le  C \left\|\sum_{k=1}^\infty \|x_k\|_2 e_{k}\right\|_{T^{(2)}}.$$
But this implies that the the natural basis of $X_J:=(\sum_{k=1}^\infty  \oplus \ell_2^{m^J_k})_{T^{(2)}}$
is equivalent to
 $\{e_{i}\}_{i \in A^J}$ and hence also equivalent to $\{e_i\}$.
 In particular, $X_J$ is isomorphic to $T^{(2)}$ and therefore may be regarded as a renorming of $T^{(2)}$. Since $T^{(2)}$ does not contain an isomorphic copy of $\ell_2$,
arguing as in Proposition~\ref{Brenormings},
$\Isom(X_J)$ is not contained in any maximal isometry group, and, arguing  as in Proposition~\ref{conjugate}, if $J \ne J'$ then
$\Isom(X_J)$ and $\Isom(X_{J'})$ are not conjugate in the isomorphism group of $T^{(2)}$.
\end{proof}

We do not know whether or not  $T^{(2)}$  admits an equivalent maximal norm.

Let $U$ be the space with a universal unconditional basis  constructed by Pe\l czy\'nski \cite{P69}.  We finish the paper by observing that both Proposition~\ref{gencontinuum} and Theorem~\ref{Bnonmax} hold for $U$. 

\begin{theorem} The  space $U$ with a universal unconditional basis   has two continua of renormings whose isometry groups are not pairwise conjugate in the isomorphism group of $U$ 
such that the renormings of the first continuum  are maximal and for the renormings of the second continuum no isometry group is contained in  any maximal bounded subgroup of $\GL(U)$. \end{theorem} 

\begin{proof}
  It is known that $U$ has a symmetric basis $E = (e_i)_{i=1}^\infty$ \cite[p.\ 129]{LT}.
  By renorming $U$, we may assume that 
 $E$ is pure. 
 To see this, let $B$ be the unit ball of any norm on $U$ for which   $E=(e_i)_{i=1}^\infty$ is a normalized $1$-symmetric basis.  Let
$$ B_1 = \overline\conv\{B \cup \{e_i \pm e_j \colon i \ne j\}\}.$$
Then $B_1$ is the unit ball for an equivalent norm $\|\cdot\|$ on $U$ such that $(e_i)_{i=1}^\infty$ is a $1$-symmetric basis satisfying
$$ \|e_i\| = \|e_i \pm e_j\| = 1 \qquad(i \ne j).$$ 
In particular, for all $i \ne j$,  $(e_i, e_j)$ is not isometric to the unit basis of $\ell_2^2$, so $E$ is pure.

The universality property of $U$ implies that 
$U(U)$ is isomorphic to $U$ (see e.g. \cite{Read})
and that $U$ is isomorphic to every space of the form
\begin{equation*} \lb{eq: anysum}
Z=\Big(\sum_{n=1}^\infty\oplus H_n\Big)_{E}, 
\end{equation*}
where $(H_n)_{n=1}^\infty$ is any collection of finite-dimensional Hilbert spaces (see e.g. \cite[p. 93]{LT}).
From this it follows that conditions of Remark~\ref{gennonmax} are satisfied, and thus Theorem~\ref{Bnonmax} is valid for $U$.

 $U$ (with the symmetric basis $E$) satisfies the hypotheses of  Proposition~\ref{gencontinuum} and hence  its conclusion holds.  
However, we prefer to give a more direct argument which also shows that the analogue of  Proposition~\ref{Brenormings} (with $E_p$ replaced by $E$ in the definition of $Y_\mathcal{B}$) holds for $U$. So consider a renorming $(\widetilde Z,\|\cdot\|_0)$ of any space $Z$ of the form as above, such that $\Isom(\widetilde Z, \|\cdot\|_0)\supseteq\Isom(Z)$. 
Then, arguing as in the proof of Theorem~\ref{nmax}, the Hilbert components of $\widetilde Z$ are spans of unions of Hilbert components of $Z$, that is,  $\bbN$ can be partitioned into disjoint subsets $\{N_j\}_{j\in J}$, so that every  Hilbert component of $\widetilde{Z}$ is given by
$$\widetilde{H}_j=\Big(\sum_{k\in N_j} \oplus H_k\Big)_2,$$
and
$$\widetilde{Z}= \Big(\sum_{j\in J} \oplus \widetilde{H}_j\Big)_{\{\widetilde{e}_j\}_{j\in J}},$$
for some pure 1-unconditional basis $\{\widetilde{e}_j\}_{j\in J}$.
Since the symmetric basis $E$ is not equivalent to the unit vector basis if $\ell_2$ it follows that the constant of equivalence between
 $(e_i)_{i=1}^n$ and the standard basis of $\ell_2^n$ becomes unbounded as $n \rightarrow \infty$.
But this implies that the sets $N_j$ 
have uniformly bounded size, i.e. 
\begin{equation*} \lb{eq: finitesizes} \max_{j \in J} |N_j| < \infty, \end{equation*}
for otherwise the norms of $Z$ and $\widetilde Z$ would not be equivalent.
The latter shows that the analogue of Proposition~\ref{Brenormings} holds for $U$.
\end{proof}
\begin{ack} The first named author was partially supported by NSF grant DMS1101490.
The second named author thanks the organizers and the Banff International Research Station  for an invitation and financial support to participate in the BIRS Workshop on Banach Spaces  in March 2012 where the work on this project  started. We also thank Valentin Ferenczi,  William B. Johnson, Denka Kutzarova and Narcisse Randrianantoanina for providing  pertinent mathematical information, and Richard J. Fleming for giving us a copy of \cite{Wood}.
\end{ack}


\begin{thebibliography}{10}


\bibitem{A} {\sc D.~J. Aldous}, {\em Subspaces of $L^1$, via random measures}, Trans. Amer. Math. Soc. 267 (1981), 445--463.

\bibitem{Banach}
{\sc S.~Banach}, {\em Th\'eorie des op\'erations lin\'eaires}, Warsaw, 1932.



\bibitem{BGRP2002}
{\sc J.~Becerra~Guerrero and A.~Rodr{\'{\i}}guez-Palacios}, {\em Transitivity
  of the norm on {B}anach spaces}, Extracta Math., 17 (2002), pp.~1--58.



\bibitem{BCLT}
{\sc J.~Bourgain, P.~G. Casazza, J.~Lindenstrauss, and L.~Tzafriri}, {\em
  Banach spaces with a unique unconditional basis, up to permutation}, Mem.
  Amer. Math. Soc., 54 (1985), pp.~iv+111.

\bibitem{CS98}
{\sc F.~Cabello~S{\'a}nchez}, {\em Maximal symmetric norms on {B}anach spaces},
  Math. Proc. R. Ir. Acad., 98A (1998), pp.~121--130.

\bibitem{CJT84}
{\sc P.~G. Casazza, W.~B. Johnson, and L.~Tzafriri}, {\em On {T}sirelson's
  space}, Israel J. Math., 47 (1984), pp.~81--98.

\bibitem{CS89}
{\sc P.~G. Casazza and T.~J. Shura}, {\em Tsirel\cprime son's space}, vol.~1363
  of Lecture Notes in Mathematics, Springer-Verlag, Berlin, 1989.
\newblock With an appendix by J. Baker, O. Slotterbeck and R. Aron.

\bibitem{CKKM} {\sc P.~G Casazza, N.~J. Kalton, Denka Kutzarova, and M. Masty\l o}, {\em Complex interpolation and complementably minimal spaces.} (English summary)  Interaction between functional analysis, harmonic analysis, and probability (Columbia, MO, 1994),  135�-143, 
Lecture Notes in Pure and Appl. Math., 175, Dekker, New York, 1996. 

\bibitem{DGZ}
{\sc R.~Deville, G.~Godefroy, and V.~Zizler}, {\em Smoothness and renormings in
  {B}anach spaces}, vol.~64 of Pitman Monographs and Surveys in Pure and
  Applied Mathematics, Longman Scientific \& Technical, Harlow, 1993.

\bibitem{D61} {\sc A. Dvoretzky}, {\em Some results on convex bodies and
  {B}anach spaces}, 1961, Proc. Internat. Sympos. Linear Spaces (Jerusalem, 1960) pp. 123--160 Jerusalem Academic Press, Jerusalem; Pergamon, Oxford.

\bibitem{FHHSPZ}{\sc Mari\'an Fabian, Petr Habala, Petr H\'ajek, Vicente Montesinos Santaluc\'ia, Jan Pelant, and V\'aclav Zizler},
{\em Functional analysis and infinite-dimensional geometry}, CMS Books in Mathematics/Ouvrages de Math\'ematiques de la SMC, 8, Springer-Verlag, New York,  2001.

\bibitem{FRmax}
{\sc V.~Ferenczi and C.~Rosendal}, {\em On isometry groups and maximal
  symmetry}, Duke Math. J. 162 (2013), no. 10, pp.~1771--1831.


\bibitem{FR2011}
\leavevmode\vrule height 2pt depth -1.6pt width 23pt, {\em Displaying {P}olish
  groups on separable {B}anach spaces}, Extracta Math., 26 (2011),
  pp.~195--233.

\bibitem{FJ74}
{\sc R.~J. Fleming and J.~E. Jamison}, {\em Isometries on certain {B}anach
  spaces}, J. London Math. Soc. (2), 9 (1974/75), pp.~121--127.

\bibitem{FJiso2}
\leavevmode\vrule height 2pt depth -1.6pt width 23pt, {\em Isometries on {B}anach spaces.
  {V}ol. 2}, vol.~138 of Chapman \& Hall/CRC Monographs and Surveys in Pure and
  Applied Mathematics, Chapman \& Hall/CRC, Boca Raton, FL, 2008.
\newblock Vector-valued function spaces.

\bibitem{F75}
{\sc Y.~Friedman}, {\em Subspace of {${\rm LC}(H)$} and {$C_{p}$}}, Proc. Amer.
  Math. Soc., 53 (1975), pp.~117--122.

\bibitem{HJX10}
{\sc U.~Haagerup, M.~Junge, and Q.~Xu}, {\em A reduction method for
  noncommutative {$L_p$}-spaces and applications}, Trans. Amer. Math. Soc., 362
  (2010), pp.~2125--2165.

\bibitem{JO} {\sc W.~B. Johnson and E.~Odell}, {\em Subspaces of $L_p$ which embed into $\ell_p$}, 
Compositio Math. 28 (1974), pp.~37--49. 

\bibitem{JZ} {\sc W.~B. Johnson and M.~Zippin}, {\em Subspaces and quotient spaces of $(\sum G_n)_{\ell_p}$ and
$(\sum G_n)_{c_0}$}, Israel J. Math. 17 (1974), pp.~50--55. 

\bibitem{JP08}
{\sc M.~Junge and J.~Parcet}, {\em Rosenthal's theorem for subspaces of
  noncommutative {$L_p$}}, Duke Math. J., 141 (2008), pp.~75--122.


\bibitem{KP}
{\sc M.~I. Kadets and A. Pe\l czy\'nski}, {\em Bases, lacunary sequences and complemented subspaces in the spaces
$L_p$}, Studia Math. 21 (1961/1962), pp.~161--176.

\bibitem{KW76}
{\sc N.~J. Kalton and G.~V. Wood}, {\em Orthonormal systems in {B}anach spaces
  and their applications}, Math. Proc. Cambridge Philos. Soc., 79 (1976),
  pp.~493--510.

\bibitem{KOS}
{\sc H.~Knaust, E.~Odell, and T.~Schlumprecht}, {\em On asymptotic structure,
  the {S}zlenk index and {UKK} properties in {B}anach spaces}, Positivity, 3
  (1999), pp.~173--199.
\bibitem{KM} {\sc J.~L. Krivine and B.~Maurey}, {\em Espaces de Banach stables}, Israel J. Math. 39 (1981), 273--295.

\bibitem{LT71} {\sc J.~Lindenstrauss and L.~Tzafriri}, {\em On Orlicz sequences spaces}, Israel J. Math., 10 (1971), 371--390.



\bibitem{LT73} \leavevmode\vrule height 2pt depth -1.6pt width 23pt, {\em On Orlicz sequences spaces, III}, Israel J. Math., 14 (1973), 368--389.


\bibitem{LT}
\leavevmode\vrule height 2pt depth -1.6pt width 23pt, {\em Classical {B}anach spaces. {I}},
  Springer-Verlag, Berlin, 1977.
\newblock Sequence spaces, Ergebnisse der Mathematik und ihrer Grenzgebiete,
  Vol. 92.

\bibitem{L79}
{\sc W.~Lusky}, {\em A note on rotations in separable {B}anach spaces}, Studia
  Math., 65 (1979), pp.~239--242.

\bibitem{M97}
{\sc J.~L. Marcolino~Nhany}, {\em La stabilit\'e des espaces {$L^p$}
  non-commutatifs}, Math. Scand., 81 (1997), pp.~212--218.

\bibitem{MMT}
{\sc B.~Maurey, V.~D. Milman, and N.~Tomczak-Jaegermann}, {\em Asymptotic
  infinite-dimensional theory of {B}anach spaces}, in Geometric aspects of
  functional analysis ({I}srael, 1992--1994), vol.~77 of Oper. Theory Adv.
  Appl., Birkh\"auser, Basel, 1995, pp.~149--175.

\bibitem{MP}
{\sc B.~Maurey et G.~Pisier}, {\em S\'eries de variables al\'eatoires vectorielles ind\'ependantes et propri\'et\'es g\'eom\'etriques des espaces de Banach}, Studia Math., 58 (1976), 45--90.

\bibitem{O}
{\sc E.~Odell}, {\em On subspaces, asymptotic structure, and distortion of
  {B}anach spaces; connections with logic}, in Analysis and logic ({M}ons,
  1997), vol.~262 of London Math. Soc. Lecture Note Ser., Cambridge Univ.
  Press, Cambridge, 2002, pp.~189--267.

\bibitem{P60} {\sc A.~Pe{\l}czy\'nski}, {\em  Projections in certain Banach spaces}, Studia Math., 19 (1960), 209--228.


\bibitem{P69} \leavevmode\vrule height 2pt depth -1.6pt width 23pt, {\em  Universal bases}, Studia Math., 32 (1969), 247--268.


\bibitem{PR62}
{\sc A.~Pe{\l}czy\'nski and S.~Rolewicz}, {\em Best norms with respect to
  isometry groups in normed linear spaces}, Short Communications on
  International Math. Congress in Stockholm, 7 (1962), p.~104.

\bibitem{P05}
{\sc G.~Pisier}, {\em Are unitarizable groups amenable?}, in Infinite groups:
  geometric, combinatorial and dynamical aspects, vol.~248 of Progr. Math.,
  Birkh\"auser, Basel, 2005, pp.~323--362.

\bibitem{R08}
{\sc N.~Randrianantoanina}, {\em Embeddings of non-commutative {$L^p$}-spaces
  into preduals of finite von {N}eumann algebras}, Israel J. Math., 163 (2008),
  pp.~1--27.

\bibitem{RX03}
{\sc Y.~Raynaud and Q.~Xu}, {\em On subspaces of non-commutative
  {$L_p$}-spaces}, J. Funct. Anal., 203 (2003), pp.~149--196.

\bibitem{Read}
{\sc C.~J. Read}, {\em When {$E$} and {$E[E]$} are isomorphic}, in Geometry of
  {B}anach spaces ({S}trobl, 1989), vol.~158 of London Math. Soc. Lecture Note
  Ser., Cambridge Univ. Press, Cambridge, 1990, pp.~245--252.

\bibitem{Rol}
{\sc S.~Rolewicz}, {\em Metric linear spaces}, PWN---Polish Scientific
  Publishers, Warsaw, second~ed., 1984.

\bibitem{R86}
{\sc H.~P.~Rosenthal}, {\em Functional {H}ilbertian sums}, Pacific J. Math., 124
  (1986), pp.~417--467.

\bibitem{R73}
{\sc H.~P. Rosenthal}, {\em On subspaces of {$L^{p}$}}, Ann. of Math. (2), 97
  (1973), pp.~344--373.

\bibitem{S04} {\sc B\"unyamin Sari}, {\em Envelope functions and asymptotic structures in {B}anach spaces}, Studia Math. 164 (2004), pp,~283--306.

\bibitem{T08}
{\sc J.~Talponen}, {\em A note on the class of super-reflexive almost transitive
  {B}anach spaces}, Extracta Math., 23 (2008), pp.~1--6.

\bibitem{Wood}
{\sc G.~V. Wood}, {\em Three conjectures on {B}anach space norms}.
\newblock unpublished, Edwardsville, May 2006.

\bibitem{W82}
\leavevmode\vrule height 2pt depth -1.6pt width 23pt, {\em Maximal symmetry in
  {B}anach spaces}, Proc. Roy. Irish Acad. Sect. A, 82 (1982), pp.~177--186.

\end{thebibliography}

\def\cprime{$'$}

\def\polhk#1{\setbox0=\hbox{#1}{\ooalign{\hidewidth\lower1.5ex\hbox{`}\hidewidth\crcr\unhbox0}}}

\end{document}